\documentclass[reqno,12pt]{amsart}

\usepackage{amsmath}
\usepackage{amsfonts}
\usepackage{amssymb}
\usepackage{eufrak}
\usepackage{amscd}
\usepackage{amsthm}
\usepackage{amstext}
\usepackage[all]{xy}

\newcommand{\Ker}{\operatorname{Ker}}
\newcommand{\id}{\operatorname{id}}

\newcommand{\ev}{\operatorname{ev}}

\renewcommand{\Im}{\operatorname{Im}}

   \theoremstyle{plain}
   \newtheorem{thm}{Theorem}[section]
   \newtheorem{prop}[thm]{Proposition}
   \newtheorem{lem}[thm]{Lemma}
   
   \theoremstyle{definition}
   
   \newtheorem{defn}[thm]{Definition}
   
   \theoremstyle{remark}
   
   \newtheorem{remark}[thm]{Remark}

 \numberwithin{equation}{section}

\author{V. Manuilov}

\date{}

\address{Moscow State University,
Leninskie Gory 1, Moscow, 
119991, Russia}

\email{manuilov@mech.math.msu.su}

\title{A $KK$-like picture for $E$-theory of $C^*$-algebras}

\begin{document}

\maketitle

\begin{abstract}
Let $A$, $B$ be separable $C^*$-algebras, $B$ stable. Elements of the $E$-theory group $E(A,B)$ are represented by asymptotic homomorphisms from the second suspension of $A$ to $B$. Our aim is to represent these elements by (families of) maps from $A$ itself to $B$. We have to pay for that by allowing these maps to be even further from $*$-homomorphisms. We prove that $E(A,B)$ can be represented by pairs $(\varphi^+,\varphi^-)$ of maps from $A$ to $B$, which are not necessarily asymptotic homomorphisms, but have the same deficiency from being ones. Not surprisingly, such pairs of maps can be viewed as pairs of asymptotic homomorphisms from some $C^*$-algebra $C$ that surjects onto $A$, and the two maps in a pair should agree on the kernel of this surjection. We give examples of full surjections $C\to A$, i.e. those, for which all classes in $E(A,B)$ can be obtained from pairs of asymptotic homomorphisms from $C$.

\end{abstract}

\section*{Introduction}

There are several pictures for Kasparov's $KK$-theory of $C^*$-algebras \cite{Kasparov}, \cite{Cuntz-quasi}, \cite{Blackadar-book}, \cite{Thomsen-book}, \cite{Thomsen-Duke} each emphasizing one of the faces of $KK$-cycles, i.e. of representatives of $KK$-classes. One of the reasons behind this diversity is lack of homomorphisms between $C^*$-algebras. Any $*$-homomorphism $\varphi:A\to B$ represents an element in $KK(A,B)$, but if one wants to represent arbitrary element of $KK(A,B)$ then one has to use generalized morphisms. For example, the Cuntz picture of $KK$-theory uses the  pairs $(\varphi^+,\varphi^-)$ of $*$-homomorphisms from $A$ to the multiplier algebra $MB$ of $B$ such that $\varphi^+(a)-\varphi^-(a)\in B$ for any $a\in A$ \cite{Cuntz-quasi}. 

$E$-theory of Connes and Higson \cite{Connes-Higson} shares the same problem: one has to use asymptotic homomorphisms in place of $*$-homomorphisms. Moreover, the asymptotic homomorphisms that represent elements in $E(A,B)$ have domain not in $A$ but in the suspension $SA$ over $A$ (or, even in the second suspension), and take values (if considered as $*$-homomorphisms) in $C_b(\mathbb R_+;B)/C_0(\mathbb R_+;B)$, where $\mathbb R_+=[0,\infty)$, and $C_b(\mathbb R_+;B)$ denotes the algebra of bounded $B$-valued functions on $\mathbb R_+$. 

Our aim is to represent elements of $E(A,B)$ by pairs of maps from $A$ itself to $C_b(\mathbb R_+;B)/C_0(\mathbb R_+;B)$ similarly to the Cuntz picture for $KK$-theory. We have to pay for that by allowing these maps to be even further from $*$-homomorphisms. We shall work with pairs of maps, which should be continuous, but may be not linear or not asymptotically linear, and we require that the two maps in a pair should have the same deficiency from being an asymptotic homomorphism. Namely, we show that elements of $E(A,B)$ can be represented by pairs $(\varphi^+,\varphi^-)$ of maps from $A$ to $C_b(\mathbb R_+;B)/C_0(\mathbb R_+;B)$ that satisfy
$$
M^+(a,b)\varphi^+(c_1)\cdots\varphi^+(c_n)=M^-(a,b)\varphi^-(c_1)\cdots\varphi^-(c_n)
$$  
for any $n=0,1,2,\ldots$ and any $a,b,c_0,c_1,\ldots,c_n\in A$, where $M^\pm(a,b)$ stands for either $\varphi^\pm(ab)-\varphi^\pm(a)\varphi^\pm(b)$ or $\varphi^\pm(a+b)-\varphi^\pm(a)-\varphi^\pm(b)$. Isomorphism between $E(A,B)$ and the group of homotopy classes of such pairs is proved using the explicit form of the Bott isomorphism obtained in \cite{MT}.

Not surprisingly, such pairs of maps can be viewed as pairs of asymptotic homomorphisms from some $C^*$-algebra $C$ that surjects onto $A$, and the two maps in a pair should agree on the kernel of this surjection. We give examples of full surjections $C\to A$, i.e. those, for which all classes in $E(A,B)$ can be obtained from pairs of asymptotic homomorphisms from $C$.

Note that in topological applications, $A$ is often a group $C^*$-algebra of a discrete group $G$, or its augmentation ideal. 
Our interest in representing $E$-theory classes by maps from $A$ instead of suspensions over $A$ is caused by our hope that group actions may give more direct constructions of these classes. 

Special cases $A=\mathbb C$ or $B=\mathbb K$ (the $C^*$-algebra of compact operators) were considered earlier in \cite{M-PJM} and in \cite{M-RJMP} respectively.

Similar approach can be used in $KK$-theory. We shall describe it in a separate paper.

\section{Asymptotic $KK$-cycles}

Let $A$, $B$ be $C^*$-algebras, $A$ separable, $B$ stable and $\sigma$-unital, $MB$ the multiplier algebra for $B$. Let $C_b(\mathbb R_+;MB)$ (resp. $C_{b,s}(\mathbb R_+;MB)$) denote the algebras of $MB$-valued functions on $\mathbb R_+=[0,\infty)$ continuous with respect to the norm (resp. to the strict) topology. For $t\in\mathbb R_+$, $\ev_t$ denotes the evaluation at $t$.

Let $C$ be a $C^*$-algebra (usually $C$ is either $C_{b,s}(\mathbb R_+;MB)$ or $C_b(\mathbb R_+;B)$).
Let $(\Phi^+,\Phi^-)$ be a pair of (not necessarily linear) continuous maps $\Phi^\pm:A\to C$. We denote by $E=E(\Phi^+,\Phi^-)=C^*(\Phi^+(A),\Phi^-(A))\subset C$ the $C^*$-algebra generated by all $\Phi^\pm(a)$, $a\in A$, and by $D=D(\Phi^+,\Phi^-)\subset E$ the ideal, in $E$, generated by $\Phi^+(a)-\Phi^-(a)$, $a\in A$.

For $x,y\in C_{b,s}(\mathbb R_+;MB)$ we write $x\sim y$ if $x-y\in C_0(\mathbb R_+;MB)$, i.e. if $\lim_{t\to\infty}x_t-y_t=0$. 

\begin{defn}\label{Def1}
A pair $(\Phi^+,\Phi^-)$ of continuous maps $\Phi^\pm:A\to C_b(\mathbb R_+;B)$ is a {\it compact asymptotic $KK$-cycle} from $A$ to $B$ if
\begin{itemize}
\item[(a1)]
$\Phi^\pm$ are $\mathbb C$-homogeneous and asymptotically involutive, i.e. $\Phi^\pm(\lambda a)=\lambda\Phi^\pm(a)$ for any $\lambda\in\mathbb C$ and for any $a\in A$, and $\Phi^\pm(a^*)-\Phi^\pm(a)^*\in C_0(\mathbb R_+;B)$ for any $a\in A$;
\item[(a2)]
$(\Phi^\pm(a+b)-\Phi^\pm(a)-\Phi^\pm(b))x\in C_0(\mathbb R_+;B)$ and $(\Phi^\pm(ab)-\Phi^\pm(a)\Phi^\pm(b))x\in C_0(\mathbb R_+;B)$ 
for any $a,b\in A$ and for any $x\in D$.
\end{itemize}

\end{defn}

Let $M^\pm(a,b)$ denote either $\Phi^\pm(a+b)-\Phi^\pm(a)-\Phi^\pm(b)$ or $\Phi^\pm(ab)-\Phi^\pm(a)\Phi^\pm(b)$. Note that (a2) is equivalent to 
$$
M^\pm(a,b)\Phi^+(c_1)\cdots\Phi^+(c_n)(\Phi^+(d)-\Phi^-(d))\in C_0(\mathbb R_+;B),
$$
or, equivalently,
\begin{equation}\label{as}
\lim_{t\to\infty}M^\pm_t(a,b)\Phi^+_t(c_1)\cdots\Phi^+_t(c_n)(\Phi^+_t(d)-\Phi^-_t(d))=0,
\end{equation}
for any $a,b,c_1,\ldots,c_n,d\in A$ and any $n=0,1,2,\ldots$.

\begin{defn}\label{Def2}
A pair $(\Phi^+,\Phi^-)$ of continuous maps $\Phi^\pm:A\to C_b(\mathbb R_+;B)$ {\it has the same deficiency from being a homomorphism} if 
\begin{itemize}
\item[(a1)]
$\Phi^\pm$ are $\mathbb C$-homogeneous and asymptotically involutive;
\item[(a3)]
$M^+(a,b)\Phi^+(c_1)\cdots\Phi^+(c_n)-M^-(a,b)\Phi^-(c_1)\cdots\Phi^-(c_n)\in C_0(\mathbb R_+;B)$
\newline
for any $n=0,1,2,\ldots$ and any $a,b,c_1,\ldots,c_n\in A$.
\end{itemize}
\end{defn}

\begin{prop}
Definitions \ref{Def1} and \ref{Def2} are equivalent.

\end{prop}
\begin{proof}
We write $x\sim y$ if $x-y\in C_0(\mathbb R_+;B)$.
It follows from  
$$
M^+(a,b)\sim M^-(a,b) \quad (\mbox{this\ is\ (a3)\ for\ }n=0)
$$
and
$$ 
M^+(a,b)\Phi^+(c)\sim M^-(a,b)\Phi^-(c) \quad (\mbox{this\ is\ (a3)\ for\ }n=1) 
$$
that $M^\pm(a,b)(\Phi^+(c)-\Phi^-(c))\sim 0$.
Similarly, 
$$
M^+(a,b)\Phi^+(c_1)\cdots\Phi^+(c_n)\sim M^-(a,b)\Phi^-(c_1)\cdots\Phi^-(c_n)
$$ 
and
$$
M^+(a,b)\Phi^+(c_1)\cdots\Phi^+(c_{n+1})\sim M^-(a,b)\Phi^-(c_1)\cdots\Phi^-(c_{n+1})
$$
imply $M^\pm(a,b)\Phi^\pm(c_1)\cdots\Phi^\pm(c_n)(\Phi^+(c_{n+1})-\Phi^-(c_{n+1}))\sim 0$. Any $x\in D$ can be approximated by finite sums of products $\Phi^+(c_1)\cdots\Phi^+(c_n)(\Phi^+(c_{n+1})-\Phi^-(c_{n+1}))e$, with $e\in E$, hence (a2) follows from (a3).

To prove the opposite, note that 
$$
\Phi^+(a)\Phi^+(b)-\Phi^-(a)\Phi^-(b)=(\Phi^+(a)-\Phi^-(a))\Phi^+(b)+\Phi^-(a)(\Phi^+(b)-\Phi^-(b))\in D,
$$
hence $M^+(a,b)-M^-(a,b)\in D$. Passing to adjoints, $M^+(a,b)^*-M^-(a,b)^*\in D$. For shortness' sake set $m_\pm=M^\pm(a,b)$. Since $m_\pm x\sim 0$ for $x=m_+^*-m_-^*\in D$, we have $(m_+-m_-)(m_+-m_-)^*\sim 0$. Then, by the $C^*$-property of the norm, we conclude that $m_+\sim m_-$, i.e. that 
$M^+(a,b)\sim M^-(a,b)$. Now, it is easy to check that (a3) holds for any $n$.

\end{proof}

\begin{remark}
It suffices to check (a3) in Definition \ref{Def2} only for $n=0$ and $n=1$:
$$
M^\pm(a,b)\Phi^+(c_1)\Phi^+(c_2)(\Phi^+(d)-\Phi^-(d))
$$
$$
=M^\pm(a,b)M^+(c_1,c_2)(\Phi^+(d)-\Phi^-(d))+M^\pm(a,b)\Phi^+(c_1c_2)(\Phi^+(d)-\Phi^-(d)).
$$

\end{remark}

\begin{defn}
A pair $(\Phi^+,\Phi^-)$ of continuous maps $\Phi^\pm:A\to C_{b,s}(\mathbb R_+;MB)$ is an {\it asymptotic $KK$-cycle} from $A$ to $B$ if
\begin{itemize}
\item[(a1)]
$\Phi^\pm$ are $\mathbb C$-homogeneous and asymptotically involutive;
\item[(b2)]
$D\subset C_b(\mathbb R_+;B)$;
\item[(b3)]
$(\Phi(a+b)-\Phi(a)-\Phi(b))x\in C_0(\mathbb R_+;B)$ and $(\Phi(ab)-\Phi(a)\Phi(b))x\in C_0(\mathbb R_+;B)$ 
for any $a,b\in A$ and for any $x\in D$.
\end{itemize}

\end{defn}

Two (compact) asymptotic $KK$-cycles $(\Phi^+_0,\Phi^-_0)$ and $(\Phi^+_1,\Phi^-_1)$ from $A$ to $B$ are homotopic if there is a (compact) asymptotic $KK$-cycle $(\Psi^+,\Psi^-)$ from $A$ to $IB=C([0,1];B)$ such that the obvious evaluation maps at 0 and at 1 give $(\Phi^+_0,\Phi^-_0)$ and $(\Phi^+_1,\Phi^-_1)$.

\begin{lem}\label{homotopy-property}
Let $\Psi^\pm_{t,\tau}:A\to MB$, $(t,\tau)\in\mathbb R_+\times[0,1]$, be a family of continuous maps. 
The following properties are equivalent:
\begin{enumerate}
\item
The pair of maps $\Psi^\pm:A\times\mathbb R_+\times[0,1]\to MB$ is an asymptotic $KK$-cycle from $A$ to $IB$;
\item
The maps $\Psi^\pm_{t,\tau}$ satisfy:
\begin{itemize}
\item[(h1)]
$\Psi^\pm_{t,\tau}$ are homogeneous and $\Psi^\pm_{t,\tau}(a^*)-\Psi^\pm_{t,\tau}(a)\in C_0(\mathbb R_+;IB)=C_0(\mathbb R_+\times[0,1];B)$ for any $a\in A$;
\item[(h2)]
$\Psi^\pm_{t,\tau}(a)\in C_{b,s}(\mathbb R_+;M(IB))$ for any $a\in A$;
\item[(h3)]
$\Psi^+_{t,\tau}(a)-\Psi^-_{t,\tau}(a)\in C_b(\mathbb R_+;IB)$ for any $a\in A$.
\item[(h4)]
$M^\pm_{t,\tau}(a,b)\Phi^+_{t,\tau}(c_1)\cdots\Phi^+_{t,\tau}(c_n)(\Phi^+_{t,\tau}(d)-\Phi^-_{t,\tau}(d))\in C_0(\mathbb R_+;IB)=C_0(\mathbb R_+\times[0,1];B)$ for any $a,b,c_1,\ldots,c_n,d\in A$;
\end{itemize}
\end{enumerate}
\end{lem}
\qed



Denote the homotopy class of a (compact) asymptotic $KK$-cycle $(\Phi^+,\Phi^-)$ by $[(\Phi^+,\Phi^-)]$, and 
denote the set of all homotopy classes of (compact) asymptotic $KK$-cycles from $A$ to $B$ by $EL(A,B)$ (resp. by $EM(A,B)$). 

Stability of $B$ allows to define addition on $EL(A,B)$ and on $EM(A,B)$: let $s_1,s_2\in MB$ be isometries with $s_1s_1^*+s_2s_2^*=1$, and let $\tilde s_1,\tilde s_2\in C_{b,s}(\mathbb R_+;MB)$ be the constant functions with values $s_1$ and $s_2$ respectively. Let $(\Phi^+,\Phi^-)$, $(\Psi^+,\Psi^-)$ be two (compact) asymptotic $KK$-cycles. Set 
$$ 
[(\Phi^+,\Phi^-)]+[(\Psi^+,\Psi^-)]=[(\tilde s_1\Phi^+\tilde s_1^*+\tilde s_2\Psi^+\tilde s_2^*,\tilde s_1\Phi^-\tilde s_1^*+\tilde s_2\Psi^-\tilde s_2^*)].
$$

It is easy to see that the pair in the right-hand side is an asymptotic $KK$-cycle, and that the addition is well defined.

\begin{lem}\label{zero}
An asymptotic $KK$-cycle $(\Phi^+,\Phi^-)$ with $\Phi^+=\Phi^-$ represents the zero element in $EL(A,B)$.

\end{lem}
\begin{proof}
Let $\Phi=\Phi^+=\Phi^-$.
Note that the ideal $D$ generated by  $\Phi^+(a)-\Phi^-(a)$, $a\in A$, is zero, hence $(\tau\Phi,\tau\Phi)$ is an asymptotic $KK$-cycle for any $\tau\in[0,1]$, so $[(\Phi,\Phi)]=[(0,0)]$. 

Let $(\Psi^+,\Psi^-)$ be an asymptotic $KK$-cycle. There are strictly continuous paths $s_1(\tau)$, $s_2(\tau)$, $\tau\in[0,1]$, such that, for $\tau\in(0,1]$, $s_1(\tau)$ and $s_2(\tau)$ are isometries with $s_1(\tau)s_1(\tau)^*+s_2(\tau)s_2(\tau)^*=1$, and $s_1(0)=1$, $s_2(0)=0$. 
Set 
$$
\Psi^\pm(a)(\tau)=\tilde s_1(\tau)\Psi^\pm(a)\tilde s_1(\tau)^*
$$
This is an asymptotic $KK$-cycle from $A$ to $IB$ ((h1)-(h4) trivially hold), hence
$[(\Psi^+,\Psi^-)]=[(\Psi^+,\Psi^-)]+[(0,0)]$. 

\end{proof}

The standard rotation argument shows that 
$$
[(\Phi^+,\Phi^-)]+[(\Psi^+,\Psi^-)]=[(\Psi^+,\Psi^-)]+[(\Phi^+,\Phi^-)]
$$ 
and that 
$$
[(\Phi^+,\Phi^-)]+[(\Phi^-,\Phi^+)]=0
$$ 
in $EL(A,B)$, hence $EL(A,B)$ is an abelian group.

Similar results hold for compact asymptotic $KK$-cycles.

\section{Making asymptotic $KK$-cycles compact}

There is a canonical map $j:EM(A,B)\to EL(A,B)$ induced by the inclusion $B\subset MB$. Here we construct a map in the opposite direction and show that these two maps are inverse to each other.

Let $(\Phi^+,\Phi^-)$ be an asymptotic $KK$-cycle, $[(\Phi^+,\Phi^-)]\in EL(A,B)$. Recall that $D\subset C_b(\mathbb R_+;B)$ denotes the ideal generated by $\Phi^+(a)-\Phi^-(a)$, $a\in A$, in the $C^*$-algebra $E$ generated by $\Phi^\pm(A)$. Let $(v_r)_{r\in\mathbb R_+}$ be an approximate unit in $D$, quasicentral in $E$. 

\begin{lem}\label{reparam0}
Let $\varphi:A\to E$ be a continuous map. There exists an order-preserving homeomorphism $r$ of $[0,\infty)$ (a reparametrization) such that
\begin{itemize}
\item[(r1)]
$\lim_{t\to\infty}[v_{r(t)},\varphi_t(a)]=0$ for any $a\in A$;
\item[(r2)]
$\lim_{t\to\infty}v_{r(t)}x_t-x_t=0$ for any $x\in D$.
\end{itemize} 

\end{lem}
\begin{proof}
As $v_r$ is an approximate unit for $D$, we have 
$$
\lim_{r\to\infty}\sup_{t\in\mathbb R_+}\|v_r(t)x_t-x_t\|=0
$$
for any $x\in D$. Therefore, any function $r=r(t)$ satisfies (r2) if $\lim_{t\to\infty}r(t)=\infty$.

Consider a dense sequence $a_1,a_2,\ldots\in A$ and positive numbers $\varepsilon_1,\varepsilon_2,\ldots$ such that $\lim_{n\to\infty}\varepsilon_n=0$. Inductively, for $r=n\in\mathbb N$ find $t_n$ such that $t_n\geq t_{n-1}+1$ and $\|[v_n(t),\varphi_t(a_i)]\|<\varepsilon_n$ for $i=1,2,\ldots,n$ when $t\geq t_n$. Extend the map $n\mapsto t_n$ to a homeomorphism of $[0,\infty)$ and take its inverse as $r$.

\end{proof}

As a bonus, we have $[v,\varphi(a)]\in C_0(\mathbb R_+;B)$ (i.e. $[v,\varphi(a)]$ is norm-continuous) for any $a\in A$. 

Set $v(t)=v_{r(t)}(t)$, $v\in C_b(\mathbb R_+;B)$, and $\Psi^\pm(a)=v\Phi^\pm(a) \in C_b(\mathbb R_+;B)$. Note that (a1) holds, i.e. that $(v\Phi^\pm(a))^*-v\Phi^\pm(a^*)\sim 0$. 

Then 
$$
(\Psi^+(ab)-\Psi^+(a)\Psi^+(b))(\Psi^+(c)-\Psi^-(c))
$$
$$
=v(\Phi^+(ab)-\Phi^+(a)\Phi^+(b))v(\Phi^+(c)-\Phi^-(c))
$$
$$
+v\Phi^+(a)(1-v)\Phi^+(b)v(\Phi^+(c)-\Phi^-(c)).
$$ 
Properties of $v$ imply that the both terms here lie in $C_0(\mathbb R_+;B)$.  

Similarly, one can show that 
$$
M^\pm(a,b)\Psi^\pm(c_1)\cdots\Psi^\pm(c_n)(\Psi^+(d)-\Psi^-(d))\in C_0(\mathbb R_+;B)
$$  
for any $a,b,c,d\in A$, where $M^\pm(a,b)$ is either $\Psi^+(a+b)-\Psi^\pm(a)-\Psi^\pm(b)$ or $\Psi^+(ab)-\Psi^\pm(a)\Psi^\pm(b)$, hence $(\Psi^+,\Psi^-)$ is a compact asymptotic $KK$-cycle.

Two different choices of an approximate unit $v$ and of a reparametrization give rise to two homotopic compact asymptotic $KK$-cycles. Also, if $(\Phi^+,\Phi^-)$ and $(\Xi^+,\Xi^-)$ are two homotopic asymptotic $KK$-cycles then one can find an approximate unit $v$ and a reparametrization related to this homotopy, so that $v^0=\ev_0 v$, $v^1=\ev_1 v$ are approximate units for the given two asymptotic $KK$-cycles, and $(v^0\Phi^+,v^0\Phi^-)$ and $(v^1\Xi^+,v^1\Xi^-)$ are homotopic. Thus this construction defines a map 
$$
p:EL(A,B)\to EM(A,B).
$$ 

\begin{thm}\label{Lemma7}
The map $p$ is an isomorphism.

\end{thm}
{\bf Claim 1.} $p\circ j$ is the identity map on $EM(A,B)$. 

Let $(\Psi^+,\Psi^-)$ be a compact asymptotic $KK$-cycle. Consider it as an asymptotic $KK$-cycle, and find an appropriate approximate unit $v$ such that $[(v\Psi^+,v\Psi^-)]=p([(\Psi^+,\Psi^-)])$. Since $\Psi^\pm(A)\subset C_b(\mathbb R_+;B)$, there exists $w\in C_b(\mathbb R_+;B)$ such that $0\leq w\leq 1$ and $(1-w)\Psi^\pm(a)\sim 0$ for any $a\in A$. Then $[(w\Psi^+,w\Psi^-)]=[(\Psi^+,\Psi^-)]$. Connecting $v$ and $w$ by the linear homotopy, we conclude that $p([(\Psi^+,\Psi^-)])=[(\Psi^+,\Psi^-)]$. \qed

{\bf Claim 2.} $j\circ p$ is the identity map on $EL(A,B)$.

Let $(\Phi^+,\Phi^-)$ be an asymptotic $KK$-cycle.
Consider the $C^*$-algebra $\mathcal E$ generated by $\Phi^\pm(A)$ and by constant $B$-valued functions on $\mathbb R_+$, and let $\mathcal D\subset\mathcal E$ be the ideal generated by $\Phi^+(a)-\Phi^-(a)$, $a\in A$, and by constant $B$-valued functions on $\mathbb R_+$. Let $(w_r)_{r\in\mathbb R_+}$ be an approximate unit in $\mathcal D$, quasicentral in $\mathcal E$. Set $w(t)=w_{r(t)}(t)$, $t\in\mathbb R_+$, $w\in C_b(\mathbb R_+;B)$, where $r$ satisfies $[w_{r(t)},\Phi^\pm_t(a)]\sim 0$ for any $a\in A$. Then the linear homotopy connecting $v$ and $w$ shows that the asymptotic $KK$-cycles $(v\Phi^+,v\Phi^-)$ and $(w\Phi^+,w\Phi^-)$ are homotopic. Note that $w\Phi^\pm(a)\sim w^{1/2}\Phi^\pm(a)w^{1/2}$.

\begin{lem}\label{unitary family}
Let $(\Psi^+_t,\Psi^-_t)$ be an asymptotic $KK$-cycle from $A$ to $B$, $U_t\in M(B)$ a norm-continuous family of unitaries, $t\in\mathbb R_+$. Then $(U_t\Psi^+_tU^*_t,U_t\Psi^-_tU^*_t)$ is an asymptotic $KK$-cocycle homotopic to $(\Psi^+_t,\Psi^-_t)$.

\end{lem}
\begin{proof}
Note that $(U_{\tau t}\Psi^+_tU^*_{\tau t},U_{\tau t}\Psi^-_tU^*_{\tau t})$ is an asymptotic $KK$-cycle from $A$ to $IB$ as $U_{\tau t}\in C_b(\mathbb R_+;M(IB))$. Then use contractibility of the unitary group of $M(B)$.

\end{proof}

Set 
$$
V=\left(\begin{matrix}w^{1/2}\\(1-w)^{1/2}\end{matrix}\right).
$$
$V$ is an isometry, and there is a norm-continuous family $U_t$ of unitaries such that $V\Phi^\pm V^*\oplus 0=U(\Phi^\pm\oplus 0) U^*$, hence $[(\Phi^+,\Phi^-)]=[(V\Phi^+V^*,V\Phi^-V^*)]$. Write $V\Phi^\pm V^*$ as a matrix:
$$
V\Phi^\pm(a) V^*=A^\pm=\left(\begin{matrix}a_{11}^\pm&a_{12}^\pm\\a_{21}^\pm&a_{22}^\pm\end{matrix}\right).
$$ 
Note that 
\begin{equation}\label{zero}
(a^+_{ij}-a^-_{ij})(1-w)^{1/2}\sim 0 \quad\mbox{if}\quad (ij)\neq (11), 
\end{equation}
as $wx\sim x$ for any $x\in D$. 

For $\tau\in[0,1]$ set $A^\pm(\tau)=\left(\begin{smallmatrix}1&0\\0&\tau\end{smallmatrix}\right)V\Phi^\pm(a) V^*\left(\begin{smallmatrix}1&0\\0&\tau\end{smallmatrix}\right)$.

\begin{lem}\label{lem11}
This is a homotopy connecting $(V\Phi^+ V^*,V\Phi^- V^*)$ with $(w^{1/2}\Phi^+ w^{1/2},w^{1/2}\Phi^- w^{1/2})$.

\end{lem}
\begin{proof}
Let $a,b\in A$, $c=ab$, and let $B^\pm=V\Phi^\pm (b)V^*$, $C^\pm=V\Phi^\pm (c)V^*$. As $(\Phi^+,\Phi^-)$ is an asymptotic $KK$-cycle, we have $C^+-A^+B^+\sim C^--A^-B^-$, which is equivalent to $C^+-C^-=(A^+-A^-)B^++A^-(B^+-B^-)$. Looking at the upper left corner and taking into account (\ref{zero}), we obtain $c^+_{11}-c^-_{11}\sim (a^+_{11}-a^-_{11})b^+_{11}+a^-_{11}(b^+_{11}-b^-_{11})$, whence 
\begin{equation}\label{z1}
c^+_{11}-a^+_{11}b^+_{11}\sim c^-_{11}-a^-_{11}b^-_{11}. 
\end{equation}
Replacing here $A,B,C$ by $A(\tau),B(\tau),C(\tau)$, we see that (\ref{z1}) does not depend on $\tau$.

Let $d\in A$, $D^\pm=V\Phi^\pm(d)V^*$. We know that $(C^+(\tau)-A^+(\tau)B^+(\tau))(D^+(\tau)-D^-(\tau))\sim 0$ when $\tau=1$.
Note that $D^+(\tau)-D^-(\tau)\sim \left(\begin{matrix}d^+_{11}-d^-_{11}&0\\0&0\end{matrix}\right)$,
$$
C^+(\tau)-A^+(\tau)B^+(\tau)=\left(\begin{matrix}c^+_{11}-a^+_{11}b^+_{11}-\tau^2a^+_{12}b^+_{21}&
\tau(c^+_{12}- a^+_{11}b^+_{12})-\tau^3 a^+_{12}b^+_{22}\\
\tau(c^+_{21}-a^+_{21}b^+_{11})-\tau^3a^+_{22}b^+_{12}&
\tau^2(c^+_{22}-a^+_{21}b^+_{12})-\tau^4a^+_{22}b^+_{22}
\end{matrix}\right).
$$
As $(1-w)^{1/2}(d^+_{11}-d^-_{11})\sim 0$, so $a^+_{ij}b^+_{kl}(d^+_{11}-d^-_{11})\sim 0$ for all $(ijkl)\neq (1111)$, therefore, $(D^+(\tau)-D^-(\tau))(C^+(\tau)-A^+(\tau)B^+(\tau))\sim 0$ for any $\tau$, and uniformly in $\tau$. This proves (h4). (h2) and (h3) are obvious.

\end{proof}


\begin{remark}
We did not use in Lemma \ref{lem11} that $w\in C_b(\mathbb R_+;B)$. It suffices to require that $w\in C_b(\mathbb R_+;MB))$, $w\Phi^\pm(a)\in C_b(\mathbb R_+;B)$, $[w,\Phi^\pm(a)]\sim 0$ and $w(\Phi^+(a)-\Phi^-(a))\sim \Phi^+(a)-\Phi^-(a)$ for any $a\in A$.

\end{remark}

\section{A map $EL(A,B)\to E(A,B)$}

Let $e_{ij}$, $i,j\in\mathbb Z$, denote the standard matrix units for $M(B\otimes\mathbb K)$. For an asymptotic $KK$-cycle $(\Phi^+,\Phi^-)$ set 
$$
\Phi(a)=\sum_{i\geq 0}\Phi^+(a)e_{ii}+\sum_{i<0}\Phi^-(a)e_{ii}, \quad T=\sum_i e_{i,i+1}.  
$$
Let $\widetilde E$ be the $C^*$-algebra generated by $\{\Phi(a):a\in A\}$ and by $T$, $\widetilde D=D\otimes\mathbb K$, where $D$ is the ideal in $E$ generated by $\Phi^+(a)-\Phi^-(a)$, $a\in A$, and $E$ is the $C^*$-algebra generated by $\Phi^\pm(A)$. Then $\widetilde D\subset \widetilde E\subset C_{b,s}(\mathbb R_+;MB)$, and $\widetilde D$ is an ideal in $\widetilde E$. 

Let $(v_r)_{r\in\mathbb R_+}$ be an approximate unit in $\widetilde D$, quasicentral in $\widetilde E$.

Let $\Phi_t=\ev_t\circ\Phi$,  $v_{r,t}=\ev_t(v_r)$, $M_t(a,b)=\ev_t(M(a,b))$, where $M(a,b)$ is either $\Phi(a+b)-\Phi(a)-\Phi(b)$ or $\Phi(ab)-\Phi(a)\Phi(b)$.

\begin{lem}\label{v}
There exists a reparametrization $r$ such that
\begin{itemize}
\item[(v1)]
$v\in C_b(\mathbb R_+;B\otimes\mathbb K)$;
\item[(v2)]
$f(v)x-x,xf(v)-x\in C_0(\mathbb R_+;B\otimes\mathbb K)$ for any $x\in\widetilde D$ and for any $f\in C_0(0,1]$;
\item[(v3)]
$f(v)y-yf(v)\in C_0(\mathbb R_+;B\otimes\mathbb K)$ for any $y\in \widetilde E$ and for any $f\in C[0,1]$;
\item[(v4)]
$f(v)M(a,b)\in C_0(\mathbb R_+;B\otimes\mathbb K)$ for any $a,b\in A$ and any $f\in C_0(0,1]$.

\end{itemize}
\end{lem}
\begin{proof}
Items (v1)-(v2) hold for any reparametrization. We have already shown how to satisfy (v3) in Lemma \ref{reparam0}, where a sequence $t_1,t_2,\ldots$ was constructed. To satisfy (v4), note that $M(a,b)v_r\sim 0$ for any $a,b\in A$ and any $r\in\mathbb R_+$. As $A$ is separable, fix a dense sequence $(c_n)_{n\in\mathbb N}\subset A$. For each $n\in\mathbb N$ find $T_n\in\mathbb R_+$ such that
$T_n\geq t_n$ and
$$
\sup_{t\geq T_n}\|M_t(a,b)v_r(t)\|<\frac{1}{n}
$$  
for $a,b\in\{c_1,\ldots,c_n\}$ and for any $r\in[0,n]$.
Then assigning $T_n$ to $n$, we get a reparametrization $r=r(t)$ such that $r(T_n)=n$. 

\end{proof}

For $a\in A$, $f\in C_0(0,1)$, $h_n(s)=e^{2\pi i ns}$, set 
\begin{equation}\label{alpha}
\alpha(\Phi^+,\Phi^-)_t(f\otimes h_n\otimes a)=f(v(t))T^n\Phi_t(a). 
\end{equation}

It follows from (v2)-(v4) that 
$$
\lim_{t\to\infty}\|\alpha(\Phi^+,\Phi^-)_t(fg\otimes h_{n+m}\otimes ab)-\alpha(\Phi^+,\Phi^-)_t(f\otimes h_n\otimes a)\alpha(\Phi^+,\Phi^-)_t(g\otimes h_m\otimes b)\|=0
$$
for any $f,g\in C_0(0,1)$, any $a,b\in A$ and any $n,m\in\mathbb N$, hence (\ref{alpha}) determines an asymptotic homomorphism $C_0(0,1)\otimes C(\mathbb S^1)\otimes A\to C_b(\mathbb R_+;B\otimes\mathbb K)$, which can be restricted to $C_0((0,1)^2)\otimes A$. We denote this asymptotic homomorphism by $\alpha(\Phi^+,\Phi^-)$, $\alpha(\Phi^+,\Phi^-)_t:S^2A\to C_b(\mathbb R_+;B\otimes\mathbb K)$.

It is clear that the map $\alpha:EL(A,B)\to E(A,B\otimes\mathbb K)$ doesn't depend on the choices of the approximate unit and of the approximation, as any two such choices can be connected by a linear homotopy.

\section{A map $E(A,B)\to EL(A,B)$}

Let $\varphi:S^2A\to C_b(\mathbb R_+;B)$ be a continuous map such that
$(\varphi_t)_{t\in\mathbb R_+}:S^2A\to B$ is an asymptotic homomorphism, where $\varphi_t=\ev\circ\varphi$. We identify here $S^2A$ with $C_0(\mathbb D;A)$, where $\mathbb D$ is the open unit ball on the complex plane. 

Let $(\alpha_i)_{i=0,1,2,\ldots}$ be a family of functions in $C_0[0,1)$ such that 
\begin{itemize}
\item[($\alpha$1)]
$\operatorname{supp}\alpha_0=[0,\frac{1}{2}]$ and $\alpha_0(0)=1$; 
\item[($\alpha$2)]
$\operatorname{supp}\alpha_i=[\frac{i-1}{i},\frac{i+1}{i+2}]$, $i=1,2,\ldots$;
\item[($\alpha$3)]
$\alpha_i\geq 0$ for any $i=0,1,2,\ldots$;
\item[($\alpha$4)]
$\sum_{i=0}^\infty\alpha_i^2=1$ (convergence here is pointwise).
\end{itemize}

Let $f:[0,1]\to[0,1]$ satisfy $f(1-t)=1-f(t)$ and $f(t)=0$ for $t\in[0,\frac{1}{10}]$, and let $g=\sqrt{f-f^2}$. Set 
$$
p_+=p_+(r,\phi)=\left(\begin{matrix}f(r)&g(r)e^{i\phi}\\g(r)e^{-i\phi}&1-f(r)\end{matrix}\right);\quad
p_-=p_-(r,\phi)=\left(\begin{matrix}f(r)&g(r)\\g(r)&1-f(r)\end{matrix}\right), 
$$
where $(r,\phi)$ are the polar coordinates on $\mathbb D$.
Note that $p_\pm\alpha_i\in M_2(C_0(\mathbb D))$ for any $i=0,1,2,\ldots$.

Set 
\begin{equation}\label{beta}
\Phi^\pm_t(a)=\sum_{i,j=0}^\infty \varphi_t(\alpha_i\alpha_j p_\pm\otimes a)\epsilon_{ij}\in M_2(C_{b,s}(\mathbb R_+;M(B\otimes\mathbb K))),
\end{equation}
where $\epsilon_{ij}$, $i,j=0,1,2,\ldots$, are the matrix units in $M(B\otimes\mathbb K)$ (later on, we shall use two copies of $\mathbb K$, so we have to use two different sets of matrix units --- $e_{ij}$ and $\epsilon_{ij}$).

Then the maps $\Phi^\pm_t$ are well-defined (the sum is strictly convergent), but are not necessarily asymptotic homomorphisms (because algebraic properties of $(\varphi_t)$ hold asymptotically in a non-uniform way). 

\begin{lem}
$(\Phi^+,\Phi^-)$ is an asymptotic $KK$-cycle. 
\end{lem}
\begin{proof}
(a1) is trivial. As $p_+\alpha_i=p_-\alpha_i$ for $i\geq 10$, so (b2) holds. To check (b3), let $P_n=\sum_{i=0}^n e_{ii}$. Then, as $p_\pm$ are projections, so, for any $n\in\mathbb N$ and for any $a,b\in A$, one has 
\begin{equation}\label{comp1}
\lim_{t\to\infty}P_nM^\pm_t(a,b)P_n=0
\end{equation}
where $M^\pm_t(a,b)$ is either $\Phi^\pm_t(a+b)-\Phi^\pm_t(a)-\Phi^\pm_t(b)$ or $\Phi^\pm_t(ab)-\Phi^\pm_t(a)\Phi^\pm_t(b)$. As $p_+(r,\phi)=p_-(r,\phi)$ for $r\geq \frac{9}{10}$, so 
\begin{equation}\label{comp2}
(1-P_{10})(\Phi^+_t(a)-\Phi^-_t(a))(1-P_{10})=0
\end{equation}
for any $a\in A$.
Note that $\alpha_i\alpha_j=0$ when $|i-j|\geq 2$, so $\Phi^\pm(a)$ (\ref{beta}) has tri-diagonal form, therefore, combining 
(\ref{comp1}) and (\ref{comp2}), we get
$$
(\Phi^\pm(ab)-\Phi^\pm(a)\Phi^\pm(b))\Phi^+(c_1)\cdots\Phi^+(c_n)(\Phi^+(d)-\Phi^-(d))\in C_0(\mathbb R_+;M_2(B))
$$ 
for any $a,b,c_1,\ldots,c_n,d\in A$.

\end{proof}

This construction doesn't depend, up to homotopy, on the choice of the functions $\alpha_i$, $i=0,1,2,\ldots$, $f$, and on the choice of $\varphi$ within its class $[\varphi]\in E(A,B)$, so this gives us a map $\beta(\varphi)=(\Phi^+,\Phi^-)$, $\beta:E(A,B)\to EL(A,B\otimes\mathbb K)$.

\section{Isomorphism $EL(A,B)\cong E(A,B)$}

\begin{thm}
The maps $\alpha$ and $\beta$ are isomorphisms between $E(A,B)$ and $EL(A,B)$.

\end{thm}
\begin{proof}
{\bf Claim 1.} $\beta\circ\alpha:EL(A,B)\to EL(A,B\otimes\mathbb K\otimes\mathbb K\otimes M_2)$ coincides with the isomorphism induced by an isomorphism $B\cong B\otimes\mathbb K\otimes\mathbb K\otimes M_2$.

Set
$$
\Xi^\pm_t(a)=\beta\circ\alpha(\Phi^+,\Phi^-)_t(a)=\sum_{i,j=0}^\infty \alpha(\Phi^+,\Phi^-)_t(\alpha_i\alpha_jp_\pm\otimes a)\epsilon_{ij}
$$

{\bf First step of homotopy.}

Let $w_{r(t)}=\sum_{i=0}^\infty \lambda_i(r(t))\epsilon_{ii}$, where $r$ is a reparametrization (i.e. an order-preserving homeomorphism of $[0,\infty)$), and $\lambda_i$, $i=0,1,\ldots$, are continuous functions such that
\begin{itemize}
\item[($\lambda$1)]
$0\leq\ldots\leq\lambda_i(t)\leq\lambda_{i+1}(t)\leq\ldots\leq 1$ for any $t\in[0,\infty)$ and any $i=0,1,\ldots$;
\item[($\lambda$2)]
for any $t\in[0,\infty)$ there is $N\in\mathbb N$ such that $\lambda_i(t)=0$ for any $i\geq N$;
\item[($\lambda$3)]
for any $i=0,1,\ldots$ $\lim_{t\to\infty}\lambda_i(t)=1$;
\item[($\lambda$4)]
$\lim_{t\to\infty}\lambda_{i+1}(t)-\lambda_i(t)=0$ for any $i=0,1,\ldots$;
\end{itemize}

Note that tri-diagonal structure of $\Xi^\pm$ and ($\lambda$4) imply that 
\begin{equation}\label{slow}
\lim_{t\to\infty}[w_{r(t)},\Xi^\pm_t(a)]=0
\end{equation}
for any $a\in A$ for an appropriately chosen reparametrization. In fact, any sufficiently slow reparametrization would do. We shall fix it later. 

By ($\lambda$2), $w_r(t)\Xi^\pm_t(a)$ is norm-continuous in $t$ ($\Xi^\pm$ is only strictly continuous).

Tri-diagonality of $\Xi^\pm$ and ($\lambda$4) imply that
\begin{equation}\label{slow}
\lim_{t\to\infty}[w_{r(t)},\Xi^\pm_t(a)]=0
\end{equation}
for any $a\in A$,
and ($\lambda$3) implies that  
$$
\lim_{t\to\infty}w_{r(t)}(\Xi^+_t(a)-\Xi^-_t(a))-(\Xi^+_t(a)-\Xi^-_t(a))=0
$$ 
for any $a\in A$. 
All these hold for any reparametrization $r$. Therefore, there exists a homotopy connecting $(\Xi^+,\Xi^-)$ with $(w\Xi^+,w\Xi^-)$ (and with $(w\Xi^+w,w\Xi^-w)$) as in Theorem \ref{Lemma7} (see Lemma \ref{lem11} and Remark after it).

Set $u=1-v$, $u=(u_t)_{t\in\mathbb R_+}$, where an approximate unit $(v_r)_{r\in\mathbb R_+}$ and a reparametrization $r$ are as in Lemma \ref{v}; 
set
$$
p_+(u,T)=\left(\begin{matrix}f(u)&g(u)T\\T^*g(u)&1-f(u)\end{matrix}\right),\quad p_-(u,T)=\left(\begin{matrix}f(u)&g(u)\\g(u)&1-f(u)\end{matrix}\right);
$$
and set 
$$
\Psi^\pm_t(a)=w_{r(t)}\sum_{|i-j|\leq 1} \alpha_i(u_t)p_\pm(u_t,T)\Phi_t(a)\alpha_j(u_t)\epsilon_{ij}\ w_{r(t)}.
$$

Note that 
$$
\lim_{t\to\infty}\alpha(\Phi^+,\Phi^-)_t(\alpha_i\alpha_jp_\pm\otimes a)-\alpha_i(u_t)p_\pm(u_t,T)\Phi_t(a)\alpha_j(u_t)=0
$$
for any $i,j\in\mathbb N$ and for any $a\in A$, and 
$$
\alpha(\Phi^+,\Phi^-)_t(\alpha_i\alpha_jp_\pm\otimes a)=0
$$
if $|i-j|>1$, so we have
$$
\lim_{t\to\infty}\Psi^\pm_t(a)-w_{r(t)}\Xi^\pm_t(a)w_{r(t)}=0,
$$
and the asymptotic $KK$-cycles $(\Xi^+,\Xi^-)$ and $(\Psi^+,\Psi^-)$ are homotopic.

{\bf Second step of homotopy.}

Set $\gamma(\tau)=\sqrt{(1-\tau)+\tau f(u)}$, $\delta(\tau)=\sqrt{1-\gamma^2(\tau)}$, $\tau\in[0,1]$; $s=\gamma(1)$, $c=\delta(1)$. Note that $\gamma(0)=1$, $\delta(0)=0$. Set 
$$
U_+(\tau)=\left(\begin{matrix}\gamma(\tau)&-\delta(\tau)T\\T^*\delta(\tau)&\gamma(\tau)\end{matrix}\right),\quad 
U_-(\tau)=\left(\begin{matrix}\gamma(\tau)&-\delta(\tau)\\\delta(\tau)&\gamma(\tau)\end{matrix}\right),
$$
\begin{equation}\label{psi}
\Psi^\pm(a)(\tau)=w_{r(t)}\sum_{|i-j|\leq 1}\alpha_i(u_t)U_\pm^*(\tau)p_\pm(u_t,T)\Phi_t(a)U_\pm(\tau)\alpha_j(u_t)\epsilon_{ij}\ w_{r(t)}, \quad \tau\in[0,1].
\end{equation}
Now, let us fix a reparametrization.

\begin{lem}\label{reparam}
Let $\varphi:A\to C_{b,s}(\mathbb R_+;MB)$ be a continuous map, let $v\in C_b(\mathbb R_+;MB)$ satisfy $0\leq v\leq 1$ and 
\begin{equation}\label{ascomm}
\lim_{t\to\infty}[v_t,\varphi_t(a)]=0
\end{equation}
for any $a\in A$. Let $\beta_i$, $i=0,1,\ldots$, be the functions defined by $\beta_i(x)=\alpha_i(1-x)$, $x\in[0,1]$. \\
1) There exists a reparametrization $r$ such that 
$$
\lim_{t\to\infty}[\varphi_t(a),\sum_{i=0}^\infty\beta_i^2(\tau v_t+(1-\tau)1)\lambda_i^2(r(t))]=0, \quad \tau\in[0,1],
$$ 
for any $a\in A$, uniformly in $\tau$;
2) For any reparametrization $r$ and for any $i=0,1,\ldots$ one has
$$
\lim_{t\to\infty}\beta_i(v_t)\sum_{i=0}^\infty\beta_i^2(v_t)\lambda_i^2(r(t))-\beta_i(u_t)=0.
$$
\end{lem}
\begin{proof}
1) Let $a_1,a_2,\ldots$ be a dense sequence in $A$, $\varepsilon_1>\varepsilon_2>\ldots>0$, $\lim_{j\to\infty}\varepsilon_j=0$. Set $g_\rho(x)=\sum_{i=0}^\infty \beta_i^2(x)\lambda_i^2(\rho)$, $x\in[0,1]$, $\rho\in[0,\infty)$. We may think of $\varphi(a)$ and the function $\tilde v$ given by $\tau\mapsto v(\tau)=\tau v+(1-\tau)1$ as elements of $C_b(\mathbb R_+;MB)\otimes C[0,1]$. Let $q:C_b(\mathbb R_+;MB)\otimes C[0,1]\to C_b(\mathbb R_+;MB)/C_0(\mathbb R_+;MB)\otimes C[0,1]$ be the quotient map. Then (\ref{ascomm}) implies that
$q(\varphi(a))$ and $q(\tilde v)$ commute. Then $q(\varphi(a))$ commutes with $q(g(\tilde v))=g(q(\tilde v))$ for any $g\in C[0,1]$, hence
$\lim_{t\to\infty}[\varphi_t(a),g(v_t(\tau))]=0$ for any $a\in A$ uniformly in $\tau$. 

Let $k_1,k_2,\ldots$ be a sequence of numbers, $k_i\geq 1$, $i\in\mathbb N$.
Start with $\rho=1$ and find $t_1\in[0,\infty)$ such that $[g_1(v_t(\tau)),\varphi_t(a_1)]<\varepsilon_1$ for any $t\geq t_1$ and any $\tau\in[0,1]$. Then take $\rho=2$ and find $t_2\in[0,\infty)$ such that $t_2\geq t_1+k_1$ and $[g_2(v_t(\tau)),\varphi_t(a_i)]<\varepsilon_2$ for $i=1,2$, for any $t\geq t_2$ and for any $\tau\in[0,1]$. Proceeding infinitely, we obtain a map $\rho\mapsto t_\rho$ when $\rho\in\mathbb N$. Extend this map to a homeomorphism on $[0,\infty)$, and take its inverse as $r$.

2) $\beta_i(v_t)\sum_{i=0}^\infty\beta_i^2(v_t)\lambda_i^2(r(t))=\beta_i(v_t)g_{r(t)}(v_t)$, so the claim follows from the fact that $g_r$ is an approximate unit in $C_0(0,1]$.

\end{proof}

When finding a reparametrization, we should keep in mind that it should be sufficiently slow to satisfy also (\ref{slow}). This can be done by adjusting the numbers $k_i$, $i\in\mathbb N$.

Note that 
\begin{equation}\label{u_contin}
U_\pm\in C_b(\mathbb R_+;IMB)
\end{equation}
 and
\begin{equation}\label{unitary}
U_\pm^*U_\pm-1\sim 0,\quad U_\pm U_\pm^*-1\sim 0
\end{equation}
uniformly in $\tau$, as $[T,\gamma(\tau)]\sim 0$, $[T,\delta(\tau)]\sim 0$ uniformly in $\tau\in[0,1]$ by (v3).

\begin{lem}
(\ref{psi}) is a homotopy of asymptotic $KK$-cycles.

\end{lem}
\begin{proof}
We have to check (h1)-(h4) from Lemma \ref{homotopy-property}.
(h1) is trivial. (h2) and (h3) follow from (\ref{u_contin}). To check (h4), set $M^\tau_\pm(a,b)=\Psi^\pm(ab)(\tau)-\Psi^\pm(a)(\tau)\Psi^\pm(b)(\tau)$, $M_\pm(a,b)=\Phi_\pm(ab)-\Phi_\pm(a)\Phi_\pm(b)$, $a,b\in A$. Let $c\in A$. By (\ref{unitary}) and Lemma \ref{reparam}, using that $\lim_{t\to\infty}[\alpha_i(u_t),U_\pm(\tau)]=0$ uniformly in $\tau$ for any $i=0,1,\ldots$, $\lim_{t\to\infty}[\alpha_i(u_t),\Phi^\pm_t(a)]=0$, and $\lim_{t\to\infty}[p_\pm(u_t,T),\Phi_t(a)]=0$ we have
$$
M^\tau_+(a,b)(\Psi^+(c)-\Psi^-(c))\sim w\sum_{|i-j|\leq 1}\alpha_i(u)U_+(\tau)p_+(u,T)M_+(a,b)U_+^*(\tau)
$$ 
\begin{equation}\label{second}
\cdot(U_+(\tau)p_+(u,T)\Phi(c)U_+^*(\tau)-U_-(\tau)p_-(u,T)\Phi(c)U_-^*(\tau))\alpha_j(u)\epsilon_{ij}\ w.
\end{equation}
Let us write the second multiplicand (\ref{second}) as
$$
U_+(\tau)p_+(u,T)\Phi(c)U_+^*(\tau)-U_-(\tau)p_-(u,T)\Phi(c)U_-^*(\tau)
$$ 
$$
=U_+(\tau)(p_+(u,T)-p_-(u,T))\Phi(c)U_+^*(\tau)
$$
$$
+U_+(\tau)p_-(u,T)\Phi(c)(U_+^*(\tau)-U_-^*(\tau))
$$
$$
+(U_+(\tau)-U_-(\tau))p_-(u,T)\Phi(c)U_-^*(\tau).
$$
Then it would suffice to show that $M_+(a,b)(p_+(u,T)-p_-(u,T))$, $M_+(a,b)p_-(u,T)\Phi(c)(U_+^*(\tau)-U_-^*(\tau))$ and 
$M_+(a,b)U_+^*(\tau)(U_+(\tau)-U_-(\tau))$ vanish as $t\to\infty$ uniformly in $\tau$.

Note that $p_+(u,T)-p_-(u,T)=\left(\begin{matrix}0&g(u)(T-1)\\(T^*-1)g(u)&0\end{matrix}\right)$, where $g(0)=g(1)=0$, hence, by (v4) $M_+(a,b)g(u)\in C_0(\mathbb R_+;B)$.

To deal with the two remaining terms, note that 
$$
U_+(\tau)-U_-(\tau)=\left(\begin{matrix}0&-\delta(\tau)(T-1)\\(T^*-1)\delta(\tau)&0\end{matrix}\right),
$$
and that $\delta(\tau)=\sqrt{1-\gamma^2(\tau)}=\sqrt{\tau(1-f(u))}=\sqrt{\tau f(v)}$, so $M_+(a,b)\delta(\tau)\sim 0$ by (v4) uniformly in $\tau$.

As
$$
p_-(u,T)(U_+^*(\tau)-U_-^*(\tau))=\left(\begin{matrix}-g(u)(T^*-1)\delta(\tau)&f\delta(\tau)(T-1)\\
-(1-f(u))(T^*-1)\delta(\tau)&g(u)\delta(\tau)(T-1)\end{matrix}\right)
$$
$$
\sim \left(\begin{matrix}-g(u)\delta(\tau)(T^*-1)&f\delta(\tau)(T-1)\\
-(1-f(u))\delta(\tau)(T^*-1)&g(u)\delta(\tau)(T-1)\end{matrix}\right)
$$
uniformly in $\tau$ by (v3), so, by (v4) $M_+(a,b)p_-(u,T)(U_+^*(\tau)-U_-^*(\tau))\sim 0$ uniformly in $\tau$. 
By (v4),
$$
[\Phi(c),U_+^*(\tau)-U_-^*(\tau)]\sim\left(\begin{matrix}0&\delta(\tau)[\Phi(c),T]\\\delta(\tau)[T^*,\Phi(c)]&0\end{matrix}\right)
$$
uniformly in $\tau$, so 
\begin{eqnarray*}
M_+(a,b)p_-(u,T)\Phi(c)(U_+^*(\tau)-U_-^*(\tau))&=&M_+(a,b)p_-(u,T)(U_+^*(\tau)-U_-^*(\tau))\Phi(c)\\
&+&M_+(a,b)p_-(u,T)[\Phi(c),U_+^*(\tau)-U_-^*(\tau)]\sim 0
\end{eqnarray*}
uniformly in $\tau$.

The last summand, $M_+(a,b)U_+^*(\tau)(U_+(\tau)-U_-(\tau))$, vanishes as $t\to\infty$ uniformly in $\tau$ by similar argument. So, $M^\tau_+(a,b)(\Psi^+(c)-\Psi^-(c))\sim 0$ uniformly in $\tau$. Similarly one can show that the same holds also for $M^\tau_-(a,b)(\Psi^+(c)-\Psi^-(c))$ and for $M^\tau_\pm(a,b)\Psi^\pm(c)(\Psi^+(d)-\Psi^-(d))$.


\end{proof}

Thus, we have a homotopy connecting $(\Psi^+(0),\Psi^-(0))$ with $(\Psi^+(1),\Psi^-(1))$. As $U_\pm(0)=\left(\begin{matrix}1&0\\0&1\end{matrix}\right)$, so $\Psi^\pm(0)=\Psi^\pm$. For $\tau=1$, $\gamma(1)=s=\sqrt{f(u)}$, $\delta(1)=c=\sqrt{1-f(u)}$, and both $s$ and $c$ asymptotically commute with $\Phi(a)$, i.e. $[s,\Phi(a)]\sim 0$, $[c,\Phi(a)]\sim 0$. Then 
$$
U^*_+(1)p_+(u,T)\Phi(a)U_+(1)-\left(\begin{matrix}s^2\Phi(a)+c^2T\Phi(a)T^*&sc(T\Phi(a)-\Phi(a)T)\\
sc(\Phi(a)T^*-T^*\Phi(a))&0\end{matrix}\right)\sim 0.
$$
As $s[T,\Phi(a)]\sim 0$, we can skip the off-diagonal elements. Similarly, 
$$
s^2\Phi(a)+c^2T\Phi(a)T^*=s^2(\Phi(a)-T\Phi(a)T^*)+(s^2+c^2)T\Phi(a)T^*,
$$ 
hence the first summand also vanishes as $t\to\infty$, so
$$
U^*_+(1)p_+(u,T)\Phi(a)U_+(1)-\left(\begin{matrix}T\Phi_t(a)T^*&0\\0&0\end{matrix}\right)\sim 0,
$$
therefore, $\Psi^+_t(a)(1)=w_{r(t)}\sum_{|i-j|\leq 1} \alpha_i(u_t)\left(\begin{matrix}T\Phi(a)T^*&0\\0&0\end{matrix}\right)\alpha_j(u_t)\epsilon_{ij}\ w_{r(t)}$.

Similarly, we have $\Psi^-_t(a)(1)=w_{r(t)}\sum_{|i-j|\leq 1}\alpha_i(u_t)\left(\begin{matrix}\Phi_t(a)&0\\0&0\end{matrix}\right)\alpha_j(u_t)\epsilon_{ij}\ w_{r(t)}$.

{\bf Third step of homotopy.}

Set $u(\tau)=1-((1-\tau)v+\tau 1)$. This is the linear homotopy connecting $u=u(0)=1-v$ with $0$.

Set 
$$
\Psi^+_t(a)(\tau)=w_{r(t)}\sum_{|i-j|\leq 1}\alpha_i(u_t(\tau))\left(\begin{matrix}T\Phi_t(a)T^*&0\\0&0\end{matrix}\right)\alpha_j(u_t(\tau))\epsilon_{ij}\ w_{r(t)};
$$ 
$$
\Psi^-_t(a)(\tau)=w_{r(t)}\sum_{|i-j|\leq 1}\alpha_i(u_t(\tau))\left(\begin{matrix}\Phi_t(a)&0\\0&0\end{matrix}\right)\alpha_j(u_t(\tau))\epsilon_{ij}\ w_{r(t)}.
$$ 

Norm-continuity of $\alpha_i(u_t(\tau))$ in $(t,\tau)$ for each $i=0,1,\ldots$ implies norm-continuity of $\Psi^\pm(a)$ in $\tau$ for each $t\in[0,\infty)$, and strict continuity in $t$ uniformly in $\tau$, so (h2) holds for $\Psi^\pm(a)$. (h3) 
follows from $\lim_{i\to\infty}\sup_{\tau\in[0,1]}\|\alpha_i(u(\tau)(T\Phi(a)T^*-\Phi(a))\|=0$. Finally, (h4) follows from
Lemma \ref{reparam}, so $(\Psi^+(\tau),\Psi^-(\tau))$ is an asymptotic $KK$-cycle from $A$ to $IB\otimes\mathbb K^2$.

When $\tau=1$, $u_t(1)=0$, hence $\alpha_0(u_t(1))=1$, $\alpha_i(u_t(1))=0$ for any $i\geq 1$. As $\lambda_0(t)=1$, we get
$$
\Lambda^+_t(a)=\Psi^+_t(a)(1)=\left(\begin{matrix}T\Phi_t(a)T^*&0\\0&0\end{matrix}\right)\epsilon_{00},\quad
\Lambda^-_t(a)=\Psi^-_t(a)(1)=\left(\begin{matrix}\Phi_t(a)&0\\0&0\end{matrix}\right)\epsilon_{00}.
$$

We can write $(\Lambda^+,\Lambda^-)=(T\Phi T^*\oplus 0,\Phi\oplus 0)=\left(\left(\begin{matrix}\Phi^+&0\\0&\Phi'\end{matrix}\right),\left(\begin{matrix}\Phi^-&0\\0&\Phi'\end{matrix}\right)\right)$, where $\Phi'$ is the infinite direct sum of copies of $\Phi^+$, $\Phi^-$ and $0$. By Lemma \ref{zero}, $[(\Lambda^+,\Lambda^-)]=[(\Phi^+,\Phi^-)]$.
\qed

{\bf Claim 2.} $\alpha\circ\beta:E(A,B)\to E(A,B\otimes\mathbb K\otimes\mathbb K\otimes M_2)$ coincides with the isomorphism induced by an isomorphism $B\cong B\otimes\mathbb K\otimes\mathbb K\otimes M_2$.

Let $\varphi:S^2A\to C_b(\mathbb R;B)$ represent an element of $E(A,B)$, and let $\Phi^\pm$ be defined as in (\ref{beta}). Denote $\alpha\circ\beta(\varphi)$ by $\psi:S^2A\to C_b(\mathbb R_+;B)$. Then 
$$
\psi(f\otimes h_n\otimes a)=f(1-v)T^n\Phi(a),
$$
where $\Phi=\sum_{i\geq 0}\Phi^+ e_{ii}+\sum_{i<0}\Phi^- e_{ii}$, $f\in C_0(0,1)$, $h(x)=e^{2\pi i x}$, and $v\in D\otimes\mathbb K\subset C_b(\mathbb R_+;B)$ satisfies
\begin{itemize}
\item
$\lim_{t\to\infty}f(v)d=0$ for any $d\in D\otimes\mathbb K$;
\item
$\lim_{t\to\infty}f(v)M(a,b)=0$ for any $a,b\in A$, where $M(a,b)$ is either $\Phi(ab)-\Phi(a)\Phi(b)$
or $\Phi(a+b)-\Phi(a)-\Phi(b)$.
\end{itemize}

Let $\phi^\pm(g\otimes a)=\sum_{i,j\geq 0}\varphi(\alpha_i\alpha_jg\otimes a)\epsilon_{ij}$, where $(\epsilon_{ij})_{i,j=0,1,2\ldots}$ is a set of matrix units for $M(B\otimes\mathbb K)$ and $g\otimes a\in \Sigma^2A=C(\mathbb S^2;A)$, and let $\phi^\pm_\infty(g\otimes a)=\sum_{i=0}^\infty \phi^\pm(g\otimes a)e_{ii}$. ($e_{ij}$ are the matrix units in $M(B\otimes\mathbb K\otimes\mathbb K)$ with respect to the second copy of $\mathbb K$.) Set $W=\sum_{i=0}^\infty e_{i,i+1}$. 

Let $E$ be the $C^*$-subalgebra of $C_{b,s}(\mathbb R_+;M(B\otimes\mathbb K^2))$ generated by $\phi^\pm_\infty(\Sigma^2A)$ and by $W$, and let $D\subset E$ be the ideal generated by $\phi^+_\infty(g\otimes a)-\phi^-_\infty(g\otimes a)$, $g\otimes a\in\Sigma^2A$, and by $\phi^+(\Sigma^2A)e_{00}$. Let $(w_r)_{r\in[0,\infty)}$ be a quasicentral aproximate unit in $D$. Without loss of generality, we may assume that $w_r$ is diagonal, $w_r=\sum_{i=0}^\infty w_{r;i}e_{ii}$. As before, we may find a reparametrization $r(t)$ such that, for any $f\in C_0(0,1)$ one has 
\begin{itemize}
\item[(w1)]
$\lim_{t\to\infty}f(w_{r(t),t})d(t)=0$ for any $d\in D$;
\item[(w2)]
$\lim_{t\to\infty}[f(w_{r(t),t}),e(t)]=0$ for any $e\in E$, 
\end{itemize}
where $w_{r(t),t}=w_{r(t)}(t)$.

Set $u'_t=1-\sum_{i\in\mathbb Z}w_{r(t),t;|i|}e_{ii}$, $u'=\sum_{i\in\mathbb Z}u'_{i}e_{ii}$, and 
$$
\psi'(f\otimes h_n\otimes a)=f(u')T^n\Phi(a).
$$
It is easy to check that $\psi'$ defines an element of $[[S^2A,B]]$, and the linear homotopy $\tau u+(1-\tau)u'$ shows that $\psi'$ is homotopic to $\psi$.

Let us write $\psi'$ as a matrix with respect to the decomposition $1=\sum_{i=1}^\infty e_{ii}+e_{00}+\sum_{i=-1}^{-\infty}e_{ii}$,
$$
f(u')=f\left(\begin{matrix}\sum_{i=1}^\infty u'_ie_{ii}&0&0\\0&u'_0e_{00}&0\\0&0&\sum_{i=-1}^{-\infty}u'_ie_{ii}\end{matrix}\right);
$$
$$
T=\left(\begin{matrix}\sum_{i=1}^\infty e_{i,i+1}&0&0\\e_{0,1}&0&0\\0&e_{-1,0}&\sum_{i=-2}^{-\infty}e_{i,i+1}\end{matrix}\right);
$$
$$
\Phi(a)=\left(\begin{matrix}\sum_{i=1}^\infty \phi^+(a)e_{ii}&0&0\\0&\phi^+(a)e_{00}&0\\0&0&\sum_{i=-1}^{-\infty}\phi^-(a)e_{ii}\end{matrix}\right).
$$

Set also
$$
f(u')_0=f\left(\begin{matrix}\sum_{i=1}^\infty u'_ie_{ii}&0&0\\0&0&0\\0&0&\sum_{i=-1}^{-\infty}u'_ie_{ii}\end{matrix}\right);
$$
$$
T_0=\left(\begin{matrix}\sum_{i=1}^\infty e_{i,i+1}&0&0\\0&0&0\\0&0&\sum_{i=-2}^{-\infty}e_{i,i+1}\end{matrix}\right);
$$
$$
\Phi(a)_0=\left(\begin{matrix}\sum_{i=1}^\infty \phi^+(a)e_{ii}&0&0\\0&0&0\\0&0&\sum_{i=-1}^{-\infty}\phi^-(a)e_{ii}\end{matrix}\right).
$$


\begin{lem}
$f(u')T^n\Phi(a)-f(u')_0T^n_0\Phi(a)_0\in C_0(\mathbb R_+;B\otimes\mathbb K^2)$.

\end{lem}
\begin{proof}
This follows from the definition of $u'$ (see (w2)). 

\end{proof} 

Set $\psi''(f\otimes h_n\otimes a)=f(u')_0T^n_0\Phi(a)_0$. Then $\psi''$ is homotopic to $\psi$, and can be written as a direct sum, $\psi''=\psi_1\oplus\psi_2\oplus 0$, where 
$$
\psi_1(f\otimes h_n\otimes a)=f(\tilde u)W^n\phi^+_\infty(a),
$$
$$
\psi_2(f\otimes h_n\otimes a)=f(\tilde u)(W^*)^n\phi^-_\infty(a),
$$ 
$\tilde u=\sum_{i=1}^\infty u'_ie_{ii}$. Therefore, $[\psi]=[\psi_1]+[\psi_2]=[\psi_1]-[\psi_0]$, where
$$
\psi_0(f\otimes h_n\otimes a)=f(\tilde u)W^n\phi^-_\infty(a).
$$

Let us recall the construction of the Bott isomorphism in $E$-theory. Let $u_0\in C_b(\mathbb R_+;M(B\otimes\mathbb K)$ be defined by $u_0=\sum_{i=1}^\infty a_i(t)e_{ii}$, where $t\in\mathbb R_+$ and $a_i(t)=\min\{\frac{t}{i},1\}$. For $g,g'\in C_0(\mathbb D^2)$, and for $\varphi:S^2A\to C_b(\mathbb R_+;B)$, $[\varphi]\in [[S^2A,B]]$, set 
$$
\mbox{Bott}(\varphi)(g\otimes g'\otimes a)=g'(u_0,W)\varphi_\infty(a),\qquad \mbox{Bott}(\varphi)\in[[S^4A,B]].
$$

Consider the commuting diagram
\begin{equation}\label{diagram1}
\begin{xymatrix}{
[[S^2\Sigma^2A,B]]\ar[r]^-{i^*}\ar[d]_-{\mathcal B}&[[S^4A,B]],\\
[[S^2A,B]]\ar[ur]_-{\mbox{Bott}}&
}\end{xymatrix}
\end{equation}
where $i^*$ is induced by the canonical inclusion $S^2A\subset \Sigma^2A$, $\mathcal B(\chi)(g\otimes a)=\chi(g\otimes p_+\otimes a)\oplus \chi(\overline{g}\otimes p_-\otimes a)$, where $g\in S^2\mathbb C$, $\overline{g}(re^\phi)=g(re^{-\phi})$ in polar coordinates, $p_\pm\in\Sigma^2\mathbb C$, $a\in A$.

Note that $\psi_k=\mathcal B(\chi_k)$, $k=0,1$, where
$$
\chi_0(g\otimes g'\otimes a)=g(\tilde u,W)\sum_{i,j=0}^\infty\varphi(\alpha_i\alpha_jp_+g'\otimes a)\epsilon_{ij},
$$
$$
\chi_1(g\otimes g'\otimes a)=g(\tilde u,W)\sum_{i,j=0}^\infty\varphi(\alpha_i\alpha_jp_-g'\otimes a)\epsilon_{ij},
$$
where $g'\in C(\mathbb D^2)$.
Thus, $[\psi]=[\mathcal B(\tilde\chi)]$,
where 
$$
\tilde\chi(g\otimes g'\otimes a)=g(\tilde u,W)\sum_{i,j=0}^\infty\varphi(\alpha_i\alpha_jg'\otimes a)\epsilon_{ij}.
$$
The linear homotopy connecting $\tilde u$ with $u_0$ shows that $[\psi]=[\mbox{Bott}(\hat\varphi)]$, where 
$$
\hat\varphi(g'\otimes a)=\sum_{i,j=0}^\infty\varphi(\alpha_i\alpha_jg'\otimes a)\epsilon_{ij}.
$$  
Finally, $[\hat\varphi]=[\varphi]$ by a homotopy similar to the homotopy via the functions $\alpha^\tau$ as in ($\alpha$1')-($\alpha$4'). 

So, we have $[\psi]=[\mathcal B(\tilde\chi)]$ and $[i^*(\tilde\chi)]=[\mbox{Bott}(\varphi)]$. Since Bott is an isomorphism, we conclude from (\ref{diagram1}) that $[\psi]=[\varphi]$. 


\end{proof}

\section{Surjections and asymptotic $KK$-cycles}

Let $A$ be a separable $C^*$-algebra. Consider all surjective $*$-homomorphisms $p:C\to A$, where $C$ is a separable $C^*$-algebra, with a partial order given by $(C,p)\leq (C',p')$ if there exists a surjective $*$-homomorphism $\lambda:C'\to C$ such that $p'=p\circ\lambda$. Denote the set of all such $C$ by $\mathcal E_A$. Abusing the notation, we shall write $C$ in place of $(C,p)$. The set $\mathcal E_A$ is directed. Indeed, if $p_1:C_1\to A$, $p_2:C_2\to A$ are surjections then the pull-back $C=\{(c_1,c_2):c_1\in C_1,c_2\in C_2,p_1(c_1)=p_2(c_2)\}$ obviously surjects onto $C_1$, $C_2$ and $A$, hence satisfies $C\geq C_1$, $C\geq C_2$. 

This directed set has a minimal and a maximal elements. The minimal element is $C=A$, and the maximal element was constructed in \cite{Cuntz2}, Section 2, under the name of {\it universal extension}.

Let $C\in\mathcal E_A$, and let $J=\Ker p$. Let $EN(C,J;B)$ be the semigroup of homotopy equivalence classes of pairs $(\psi^+,\psi^-)$ of asymptotic homomorphisms from $C$ to $B$ that agree on $J$, with the semigroup structure defined using stability of $B$. This semigroup need not be a group, as $(\psi^-,\psi^+)$ may not be the inverse for $(\psi^+,\psi^-)$. Set $EN(A,B)=\injlim_{C\in\mathcal E_A}EN(C,J;B)$. 

\begin{lem}
$EN(A,B)$ is a group.

\end{lem}
\begin{proof}
Due to the standard rotation argument, it suffices to prove that $(\psi,\psi)$ represents the zero element in $EN(A,B)$, where $C\in\mathcal E_A$, and $\psi$ is an asymptotic homomorphism from $C$ to $B$. Let $Cone(C)$ be the cone over $C$, $ev_1:Cone(C)\to C$ the evaluation map at 1, then $p\circ\ev_1:Cone(C)\to A$ is a surjection, so $Cone(C)\in\mathcal E_A$, $Cone(C)\geq C$. As $Cone(C)$ is contractible, $\psi\circ\ev_1$ is homotopic to zero, hence the pair $(\psi,\psi)$ is zero in $EN(A,B)$. 

\end{proof}

Let $C\in\mathcal E_A$, $(\psi^+,\psi^-)\in EN(C,J;B)$. 
Let $s:A\to C$ be a $\mathbb C$-homogeneous $*$-respecting continuous map such that $p(s(a))=a$ for any $a\in A$, which exists by \cite{Bartle-Graves}. Set $\Phi^\pm(a)=\psi^\pm(s(a))$, $a\in A$. 
\begin{lem}
The pair $(\Phi^+,\Phi^-)$ is a compact asymptotic $KK$-cycle.

\end{lem}
\begin{proof}
The property (a1) is obvious, so we need to check (a3). Note that 
$$
(\Phi^\pm(ab)-\Phi^\pm(a)\Phi^\pm(b))\Phi^\pm(c_1)\cdots\Phi^\pm(c_n)
$$
$$
=(\psi^\pm(s(ab))-\psi^\pm(s(a))\psi^\pm(s(b)))\psi^\pm(s(c_1))\cdots\psi^\pm(s(c_n))
$$ 
$$
\sim(\psi^\pm(s(ab)-s(a)s(b)))\psi^\pm(s(c_1))\cdots \psi^\pm(s(c_n))
$$
$$
\sim \psi^\pm((s(ab)-s(a)s(b))s(c_1)\cdots s(c_n)),
$$
and that $s(ab)-s(a)s(b)\in J$, so $(s(ab)-s(a)s(b))s(c_1)\cdots s(c_n)\in J$, and as $\psi^+$ and $\psi^-$ agree on $J$, so we are done. 

\end{proof}

If $s,s':A\to C$ are two maps satisfying the above assumptions then the linear homotopy connecting $s$ and $s'$ gives a homotopy between the corresponding compact asymptotic $KK$-cycles, so the above construction gives a well-defined map
$$
\gamma_C:EN(C,J;B)\to EM(A,B),\quad \gamma_C(\psi^+,\psi^-)=(\psi^+\circ s,\psi^-\circ s).
$$
Let $\kappa:C'\to C$ be a surjection, $p'=p\circ\kappa:C'\to A$, $J'=\Ker p'$. Then we have $\gamma_{C'}=\kappa^*\circ\gamma_C$, where the map $\kappa^*:EN(C,J;B)\to EN(C',J';B)$ is induced by $\kappa$, and we may pass to the direct limit.

Let us check homotopy invariance of $\gamma$. Let $[(\psi^+_i,\psi^-_i)]$, $i=0,1$, where $\psi^\pm_i$ are asymptotic homomorphisms from $C$ to $B$, represent the same element in $EN(A,B)$. This means that there exists some $C'\geq C$ in $\mathcal E_A$ and a homotopy $(\varphi^+,\varphi^-)$, where $\varphi^\pm$ are asymptotic homomorphisms from $C'$ to $IB$ such that $\ev_i\circ\varphi^\pm=\psi^\pm_i\circ\kappa$, $i=0,1$, and $\varphi^+$ coinsides with $\varphi^-$ on $\Ker(C'\to A)$. Let $s':C\to C'$ be a continuous section for the surjection $\kappa$. Then $\gamma_{C'}(\varphi^+,\varphi^-)=(\varphi^+\circ s'\circ s,\varphi^-\circ s'\circ s)$. As 
$$
\ev_i\circ\varphi^\pm\circ s'\circ s=\psi^\pm_i\circ\kappa\circ s'\circ s=\psi^\pm_i\circ s, \quad i=0,1,
$$ 
we see that the map $\gamma:EN(A,B)\to EM(A,B)$ is well defined as the direct limit of the maps $\gamma_C$. 


Now let us construct a map $\delta:EM(A,B)\to EN(A,B)$. Let $[(\Phi^+,\Phi^-)]\in EM(A,B)$. Recall that $E\subset C_b(\mathbb R_+;B)$ is the $C^*$-algebra generated by $\Phi^\pm(A)$, and that the compact asymptotic $KK$-cycle $(\Phi^+,\Phi^-)$ determines an ideal $D\subset E$ generated by $\Phi^+(a)-\Phi^-(a)$, $a\in A$. It also determines another ideal $K\subset E$ generated by $\Phi^\pm(a+b)-\Phi^\pm(a)-\Phi^\pm(b)$ and $\Phi^\pm(ab)-\Phi^\pm(a)\Phi^\pm(b)$, $a,b\in A$. By definition of asymptotic $KK$-cycles, $KD\subset C_0(\mathbb R_+;B)$, hence $K\cap D\subset C_0(\mathbb R_+;B)$. 

Let $E_+$ (resp. $E_-$) be the $C^*$-algebra generated by $\Phi^+(A)$ (resp. by $\Phi^-(A)$). Denote the quotient map $C_b(\mathbb R_+;B)\to C_b(\mathbb R_+;B)/C_0(\mathbb R_+;B)$ by $Q$, and for a $C^*$-subalgebra $C\subset C_b(\mathbb R_+;B)$ set $\underline{C}=Q(C)$. Then $\underline{D}$ and $\underline{K}$ are ideals in $\underline{E}$ with $\underline{D}\cap\underline{K}=0$, and $\underline{E}$ is generated by $\underline{E}_+$ and $\underline{E}_-$.

\begin{lem}\label{surj}
For a compact asymptotic $KK$-cycle $(\Phi^+,\Phi^-)$ there exists $[\psi]\in EN(A,B)$ such that $\gamma([\psi])=[(\Phi^+,\Phi^-)]$.

\end{lem}
\begin{proof}

Let $q:\underline{E}\to \underline{E}/\underline{K}$, $r:\underline{E}\to \underline{E}/\underline{D}$. 
Note that the compositions 
$$
\begin{xymatrix}{
\underline{\Phi}^\pm:A\ar[r]^-{\Phi^\pm}&E\ar[r]^-{Q}&\underline{E}\ar[r]^-{q}&\underline{E}/\underline{K}
}\end{xymatrix}
$$
are genuine $*$-homomorphisms.
 Set
$$
C_\Phi=\{(a,e_+,e_-):a\in A,e_\pm\in \underline{E}_\pm, \underline{\Phi}^\pm(a)=q(e_\pm),r(e_+)=r(e_-)\}.
$$
This is a $C^*$-algebra that surjects onto $A$, $p(a,e_+,e_-)=a$. It has also two surjections, $p_+(a,e_+,e_-)=e_+$, $p_-(a,e_+,e_-)=e_-$, onto $\underline{E}_+$ and $\underline{E}_-$ respectively. Let $\sigma:\underline{E}\to E$ be a continuous Bartle--Graves selection map. As $p_\pm$ are $*$-homomorphisms, $\psi^\pm=\sigma\circ p_\pm:C_\Phi\to E$ are asymptotic homomorphisms. 

Note that 
$$
J_\Phi=\Ker p=\{(0,e_+,e_-):e_\pm\in \underline{K},e_+-e_-\in\underline {D}\}=\{(0,m,m):m\in \underline{K}\},
$$ 
as $\underline{D}\cap\underline{K}=0$.
Then $\psi^+|_{J_\Phi}=\psi^-|_{J_\Phi}$. Define the map 
$$
s:A\to C_\Phi \quad\mbox{by}\quad s(a)=(a,Q(\Phi^+(a)),Q(\Phi^-(a))).   
$$
Then $s$ is a continuous right inverse for $p$, and $\psi^\pm(s(a))\sim\Phi^\pm(a)$. 

If $C\in\mathcal E_A$, $C\geq C_\Phi$, then there is a $*$-homomorphism $\kappa:C\to C_\Phi$, and we can define asymptotic homomorphisms $\psi_C^\pm:C\to B$ by $\psi_C^\pm=\psi^\pm\circ\kappa$. As $\psi_C^+|_J=\psi_C^-|_J$, where $J=\Ker p$, $p:C\to A$, so $(\psi_C^+,\psi_C^-)$ represents an element of $EN(C,J;B)$, and $\gamma([(\psi^+,\psi^-)])=[(\Phi^+,\Phi^-)]$. \qed

\end{proof}

Set $\delta(\Phi^+,\Phi^-)=(\psi^+,\psi^-)$, and let us check homotopy invariance of $\delta$.
Let $(\Psi^+,\Psi^-)$ be a compact asymptotic $KK$-cycle from $A$ to $IB$, $\Phi^\pm_i=\ev_i\circ\Psi^\pm$, $i=0,1$.
Apply the construction of Lemma \ref{surj} to this compact asymptotic $KK$-cycle. Let $E_I\subset C_b(\mathbb R_+;IB)$ be the $C^*$-algebra generated by $\Psi^\pm(A)$,and let $K_I\subset E_I$ be the ideal generated by $\Phi^\pm(a+b)-\Phi^\pm(a)-\Phi^\pm(b)$ and $\Phi^\pm(ab)-\Phi^\pm(a)\Phi^\pm(b)$, $a,b\in A$. Similarly, Let $K_i\subset E_i\subset C_b(\mathbb R_+;B)$ be the corresponding ideal and $C^*$-subalgebra for $\ev_i\circ\Phi^\pm$, and let underlining denotes the image of these $C^*$-algebras in $C_b(\mathbb R_+;IB)/C_0(\mathbb R_+;IB)$ and in $C_b(\mathbb R_+;B)/C_0(\mathbb R_+;B)$ respectively.

Let $q_I:\underline{E}_I\to\underline{E}_I/\underline{K}_I$, $q_i:\underline{E}_i\to\underline{E}_i/\underline{K}_i$, $r_I:\underline{E}_I\to\underline{E}_I/\underline{J}_I$, $r_i:\underline{E}_i\to\underline{E}_i/\underline{J}_i$ be the quotient maps, and let, as in Lemma \ref{surj}, $\underline{\Psi}^\pm:A\to \underline{E}_I/\underline{K}_I$ be the $*$-homomorphisms induced by $\Psi^\pm$. Let $C_\Psi$ and $C_{\Phi_i}$, $i=0,1$, be the $C^*$-algebras in $\mathcal E_A$ constructed as in Lemma \ref{surj}, with surjections $p_I:C_{\Psi}\to A$ and $p_i:C_{\Phi_i}\to A$. Note that $\ev_i:\underline{E}_I\to\underline{E}_i$ is surjective, and $\ev_i(\underline{K}_I)=\underline{K}_i$, hence $\ev_i$ gives rise to a surjective $*$-homomorphism $\widetilde{\ev}_i:\underline{E}_I/\underline{K}_I\to\underline{E}_i/\underline{K}_i$.

This gives a commuting diagram
$$
\begin{xymatrix}{
C_\Psi\ar[rrr]^-{p_\pm^\Psi}\ar[d]_-{\kappa_i}&&&\underline{E}_I\ar[ld]_-{\ev_i}\ar[dr]^-{q_I}&&\\
C_{\Phi_i}\ar[rr]^-{p_\pm^{\Phi_i}}\ar[d]_-{p_i}&&\underline{E}_i\ar[rd]^-{q_i}&&\underline{E}_I/\underline{K}_I\ar[dl]_-{\widetilde{\ev_i}}\\
A\ar[urrrr]_-{\underline{\Psi}^\pm}\ar[rrr]_-{\underline{\Phi}^\pm_i}&&&\underline{E}_i/\underline{K}_i,&
}\end{xymatrix}
$$
where $\kappa_i:C_\Psi\to C_{\Phi_i}$ is given by $\kappa_i(a,e_+,e_-)=(a,\ev_i(e_+),\ev_i(e_-))$. Recall that $(a,e_+,e_-)\in C_\Psi$ if $a\in A$, $e_\pm\in\underline{E}_I$, $\underline{\Psi}^\pm(a)=q_I(e_\pm)$, $r_I(e_+)=r_I(e_-)$. 

Let $\sigma_I:\underline{E}_I\to E_I$, $\sigma_i:\underline{E}_i\to E_i$. Then 
$\delta(\Psi^+,\Psi^-)=(\psi^+,\psi^-)$, where $\psi^\pm=p_\pm^\Psi\circ\sigma_I$, $\delta(\Phi^+_i,\Phi^-_i)=(\varphi^+_i,\varphi^-_i)$, where $\varphi^\pm_i=p_\pm^{\Phi_i}\circ\sigma_i$. But the asymptotic homomorphisms $\ev_i\circ\psi^\pm$ are asymptotically equivalent to $\varphi^\pm_i\circ\kappa$, $i=0,1$, so $[(\varphi^+_0,\varphi^-_0)]=[(\varphi^+_1,\varphi^-_1)]$. Thus, the map $\delta:EM(A,B)\to EN(A,B)$ is well-defined.

\begin{thm}
The map $\gamma:EN(A,B)\to EM(A,B)$ is an isomorphism.

\end{thm}
\begin{proof}
Lemma \ref{surj} proves surjectivity of $\gamma$, so it suffices to show that $\delta\circ\gamma$ is the identity map. 

Let $C\in\mathcal E_A$, $\varphi^\pm$ be asymptotic homomorphisms from $C$ to $B$ equal on $J=\Ker(C\to A)$. Then $\gamma(\varphi^+,\varphi^-)=(\Phi^+,\Phi^-)$, where the latter is a compact asymptotic $KK$-cycle defined by $\Phi^\pm=\varphi^\pm\circ s$, where $s:A\to C$ is a continuous section.
Using notation from Lemma \ref{surj}, we have two commuting diagrams
$$
\begin{xymatrix}{
C\ar[r]^-{[\varphi^\pm]}\ar[d]_-{p}&\underline{E}\ar[d]^-{q}\\
A\ar[r]_-{\underline{\Phi}^\pm}&\underline{E}/\underline{K}
}\end{xymatrix}\qquad\mbox{and}\qquad
\begin{xymatrix}{
C_\Phi\ar[r]^-{p_\pm}\ar[d]_-{p_\Phi}&\underline{E}\ar[d]^-{q}\\
A\ar[r]_-{\underline{\Phi}^\pm}&\underline{E}/\underline{K}
}\end{xymatrix}
$$
where $[\varphi^\pm]$ denotes the composition of $\varphi^\pm$ with the quotient map $Q:E\to\underline{E}$.
There exists a $*$-homomorphism $\kappa:C\to C_\Phi$ given by $\kappa(c)=(p(c),[\varphi^+](c),[\varphi^-](c))$ (as $\varphi^+$ and $\varphi^-$ are equal on $J$, $r([\varphi^+](c))=r([\varphi^-](c))$), so $\kappa(c)\in C_\Phi$. As $p$ is surjective and $[\varphi^\pm]$ are surjective too, by the definition of $\underline{E}_\pm$, $\kappa$ is surjective. Thus, $[\varphi^\pm]=\kappa\circ p_\pm$, hence, passing to their lifts, $\varphi^\pm$ is asymptotically equivalent to $\psi^\pm=\sigma\circ\kappa\circ p_\pm$, where $\sigma:\underline{E}\to E$ is a continuous lift for $Q$. But $\delta\circ\gamma(\varphi^+,\varphi^-)=(\psi^+,\psi^-)$.

\end{proof}

\section{Full surjections}

\begin{defn}
Let $C\in\mathcal E_A$. We call it {\it full} if the canonical map $EN(C,J;B)\to EN(A,B)$ is an isomorphism for any $B$. 

\end{defn}

Note that $\mathcal E_A$ has the minimal element $A$. For this minimal element, $EN(A,0;B)$ is the semigroup of homotopy classes of pairs of asymptotic homomorphisms from $A$ to $B$, hence $EN(A,0;B)\cong [[A,B]]\oplus[[A,B]]$, and thus $A$  very rarely is full. 

\begin{lem}
The cone extension $CA$ and the universal extension are full.

\end{lem}
\begin{proof}
For the universal extension the statement follows from its universality, so let us prove that $CA$ is full. Essentially, this is a special case of Theorem 5.3 of \cite{Thomsen}, but we give more details here.

We shall construct a map $\delta_{CA}:EM(A,B)\to EN(CA,SA;B)$ inverse to $\gamma_{CA}$. Given $(\Phi^+,\Phi^-)$ in $EM(A,B)$, let $v\in C_b(\mathbb R_+;B)$ be as before, i.e. let $v$ satisfy
\begin{itemize}
\item[(v1')]
$0\leq v\leq 1$;
\item[(v2')]
$[f(v),\Phi^\pm(a)]\sim 0$ for any $f\in C[0,1]$ and any $a\in A$;
\item[(v3')]
$(f(v)-1)(\Phi^+(a)-\Phi^-(a))\sim 0$ for any $f\in C_0[0,1)$ and any $a\in A$.
\end{itemize}
Define $\psi^\pm:CA\to C_b(\mathbb R_+;B)$ by $\psi^\pm(f\otimes a)=f(v)\Phi^\pm(a)$, where $f\in C_0[0,1)$, $a\in A$.

It is easy to see that $\psi^\pm$ are asymptotic homomorphisms from $CA$ to $B$, and (v3') provides that $\psi^+$ and $\psi^-$ agree on $SA$, i.e. when $f\in C_0(0,1)$. Thus we get a well-defined map $\delta_{CA}:EM(A,B)\to EN(CA,SA;B)$, $\delta_{CA}(\Phi^+,\Phi^-)=(\psi^+,\psi^-)$.

Then $\delta_{CA}\circ\gamma_{CA}(\Phi^+,\Phi^-)=(v\Phi^+,v\Phi^-)$, and, as in the proof of Lemma \ref{Lemma7}, $[v\Phi^+,v\Phi^-]=[\Phi^+,\Phi^-]$, hence $\delta_{CA}\circ\gamma_{CA}=\id$.

Let us show that $\gamma_{CA}\circ\delta_{CA}=\id$. 
Let $(\psi^+,\psi^-)\in EN(Cone(A),SA;B)$, i.e. $\psi^\pm$ are two asymptotic homomorphisms from $Cone(A)$ to $B$ that agree on $SA$. Let $s:A\to CA$ be the map defined by $s(a)(\tau)=\tau a$, $a\in A$, $\tau\in[0,1]$. Then $\gamma_{CA}\circ\delta_{CA}(\psi)=(\phi^+,\phi^-)$, where $\phi^\pm(f\otimes a)=f(v)\psi^\pm(s(a))$, $a\in A$, $f\in C_0[0,1)$. 

Let us show that the pairs $(\psi^+,\psi^-)$ and $(\phi^+,\phi^-)$ of asymptotic homomorphisms are homotopic, i.e. that there exist two homotopies, connecting $\psi^+$ with $\phi^+$, and $\psi^-$ with $\phi^-$, such that they agree on $SA$.

Define asymptotic homomorphisms $\xi^\pm$ (resp. $\eta^\pm$) from $C_0([0,1)\times[0,1])\otimes A$ (resp. from $C_0([0,1]\times[0,1))\otimes A$) to $B$ by
$$
\xi^\pm(f\otimes g\otimes a)=f(v)\psi^\pm(g\cdot s(a)), \quad f\in C_0[0,1), g\in C[0,1], a\in A;
$$  
$$
\eta^\pm(f\otimes g\otimes a)=f(v)\psi^\pm(g\otimes a), \quad f\in C[0,1], g\in C_0[0,1), a\in A.
$$  

Note that on the common domain $C_0([0,1)^2\otimes A)$, these asymptotic homomorphisms are homotopic via the linear homotopy
$$
\zeta^\pm_\tau(f\otimes g\otimes a)\to f(v)\psi^\pm(\tau g\otimes a+ (1-\tau)g\cdot s(a)),\quad f,g\in C_0[0,1), \tau\in[0,1],
$$
and if $g\in C_0(0,1)$ then the two maps $\zeta_\tau^+$ and $\zeta_\tau^-$ agree.

Let $(x,y)$ be coordinates in $[0,1]^2$, and let $p_1,p_2,p:[0,1]^2\to[0,1]$ be defined by $p_1(x,y)=x$, $p_2(x,y)=y$, $p(x,y)=\left\lbrace\begin{array}{cc}x+y&\mbox{\ if\ }x+y\leq 1;\\1&\mbox{\ if\ }x+y\geq 1.\end{array}\right.$   
There are obvious linear homotopies $p_1(\tau)$ and $p_2(\tau)$, $\tau\in[0,1$, connecting $p_1$ and $p_2$ with $p$. 
By $q^*$ we denote the $*$-homomorphism $C[0,1)\otimes A\to C[0,1]^2\otimes A$ induced by a map $q:[0,1]^2\to[0,1]$.

Note that $p_1^*\circ\xi^\pm(f\otimes a)=f(v)\psi^\pm(s(a))=\varphi^\pm(f\otimes a)$, and $\Im p_1^*=C_0([0,1)\times[0,1])$. The maps $p_1^*\circ\xi^\pm$ and $p_2^*\circ\xi^\pm$ are homotopic via the homotopy $p_1^*(\tau)\circ\xi^\pm$, $\tau\in[0,1]$, and $p_1^*(\tau)\circ\xi^+$ and $p_1^*(\tau)\circ\xi^-$ agree on $f\otimes a$ when $f\in C_0(0,1)$, for any $\tau\in[0,1]$. Hence $[(\varphi^+,\varphi^-)]=[(p^*\circ\xi^+,p^*\circ\xi^-)]$.

Similarly, $p_2^*\circ\eta^\pm(f\otimes a)=\psi^\pm(f\otimes a)$, and $[(\psi^+,\psi^-)]=[(p^*\circ\eta^+,p^*\circ\eta^-)]$.


As $p^*(f\otimes a)\in C_0([0,1)^2)$ for any $f\in C_0[0,1)$, we may use the homotopy $\zeta^\pm$ to connect $p^*\circ\xi^\pm$ with $p^*\circ\eta^\pm$. 

\end{proof}

\begin{lem}
Let $C\in\mathcal E_A$. If there exists a splitting $*$-homomorphism $\lambda:A\to C$ and if $A\in\mathcal E_A$ is not full then $C$ is not full.

\end{lem}
\begin{proof}
Any continuous map $s:A\to C$ can be connected to $\lambda$ by a linear homotopy, therefore, any compact asymptotic $KK$-cycle is homotopic to a compact asymptotic $KK$-cycle, which is a pair of $*$-homomorphisms, hence $EN(C,J;B)\cong EN(A,0;B)$.

\end{proof}

\end{document}